\documentclass[10pt]{article}

\usepackage[T1]{fontenc}
\usepackage[utf8]{inputenc}
\usepackage{amsmath,amssymb,amsthm,mathtools}
\usepackage{txfonts}
\usepackage{enumitem}
\usepackage[
  paperwidth=210mm,
  paperheight=285mm,
  left=15mm,
  right=15mm,
  top=25mm,
  bottom=19.5mm
]{geometry}
\usepackage[hidelinks]{hyperref}
\usepackage{titlesec}

\allowdisplaybreaks
\titleformat{\section}{\bfseries\normalsize}{\thesection.}{0.45em}{}
\titleformat{\subsection}{\bfseries\small}{\thesubsection}{0.45em}{}
\titlespacing*{\section}{0pt}{0.75em}{0.25em}
\titlespacing*{\subsection}{0pt}{0.65em}{0.20em}

\newtheoremstyle{published}%
  {5pt}{5pt}{\itshape}{}
  {\bfseries}{}{0.5em}
  {\thmname{#1}\thmnumber{ #2}\thmnote{ #3}}
\theoremstyle{published}
\newtheorem{theorem}{Theorem}
\newtheorem{proposition}{Proposition}
\newtheorem{lemma}{Lemma}

\newcommand{\R}{\mathbb{R}}
\newcommand{\Om}{\Omega}
\newcommand{\Gam}{\Gamma}
\newcommand{\norm}[2]{\left\lVert #1\right\rVert_{#2}}
\newcommand{\abs}[1]{\left\lvert #1\right\rvert}
\newcommand{\dd}{d}
\newcommand{\diver}{div}

\begin{document}

{\centering
\fontsize{17.25}{22.46}\selectfont
A Priori and A Posteriori Error Estimates for a Crank Nicolson Type\\
Scheme of an Elliptic Problem with Dynamical Boundary Conditions\par}

\vspace{8pt}
{\centering
Rola Ali Ahmad\textsuperscript{1}, Toufic El Arwadi\textsuperscript{1}, Houssam Chrayteh\textsuperscript{1} \& Jean-Marc Sac-Ep\'ee\textsuperscript{2}\par}

\vspace{5pt}
\noindent\textsuperscript{1} Beirut Arab University, Lebanon\\[5pt]
\textsuperscript{2} Lorraine University, France\\[5pt]
Correspondence: Toufic El Arwadi, Beirut Arab University, Lebanon. E-mail: t.elarwadi@bau.edu.lb

\vspace{13pt}
\noindent\textbf{Abstract}

In this article we claim that we are going to give a priori and a posteriori error estimates for a Crank Nicolson type scheme.
The problem is discretized by the finite elements in space. The main result of this paper consists in establishing two types
of error indicators, the first one linked to the time discretization and the second one to the space discretization.


\section{Introduction}

Let $\Om$ be a bounded smooth sub domain of $\R^2$ and
$\gamma(x)=[\gamma_{i,j}]_{i,j=1}^{n}$ be a real positive definite matrix-valued function. Let $(0,T)$ denote a subinterval
of $\R$ where $T\in(0,\infty)$ is a fixed final time. Denote by $n(x)$ the unit outward normal vector at
$x\in\Gam$. We intend to work with the following problem,
\begin{equation}
\begin{cases}
\diver(\gamma\nabla u)=0, & \text{in }(0,T)\times\Om\\
\dfrac{\partial u}{\partial t}(t,x)+\gamma n(x).\nabla u(t,x)=0,
   & \text{on }(0,T)\times\Gam\\
u(0,x)=u_0(x), & \text{on }\Gam
\end{cases}
\label{eq:continuous-problem}
\end{equation}
where $\Gam$ is the boundary, $u$ is the unknown and $u_0$ is the initial condition at time $t=0$.

The solution of the above problem can be represented on the boundary by the Dirichlet-to-Neumann semigroup
(Vrabie, 2003) defined as
\[
  (S (t)f) (x)=u(t,x)|_{\Gam}
\]

In (Cherif, Arwadi, Emmamirad \& Sac Epee, 2014), the authors showed that the Lax semigroup is the
Dirichlet-to-Neumann semigroup in the particular case where $\Om=B (0,1)$ is the unit ball of $\R^2$ and $\gamma(x)$
is the identity matrix. P. Lax showed in his book (Lax, 2002) that the DtN semigroup has an explicit representation.
This was a motivation for the authors in (Cherif, Arwadi, Emmamirad \& Sac Epee, 2014) and
(Emmamirad \& Shariftabbar, 2013) to introduce semi discrete implicit and explicit Euler's schemes to approximate
the DtN semigroup numerically. They also showed the convergence of these schemes using the Chernoff's product formula.

For more than twenty years, an impressive amount of work has been accomplished concerning a posteriori analysis and
mesh adaptivity for the finite element discretization of the elliptic problems. Their main results were to exhibit local error
indicators which can be computed explicitly as a function of the discrete solution and the data.

In (Arwadi, Dib \& Sayah, 2015), they studied the time dependent linear elliptic problem, and established optimal a priori
and a posteriori error estimates using the backward Euler's scheme in time and finite elements in space.

The Crank Nicolson scheme is one of the most popular time-stepping method; however optimal a priori and a posteriori
error estimates for elliptic equations have not yet been derived. The aim of this work is to provide optimal a priori and
a posteriori estimates and some numerical investigations.

The term ``a posteriori error estimator'' was first used by Ostrowski (Ostrowski, 1940). It is the quantity which bounds or
approximates the error, i.e. an upper bound of the error between an exact solution and a numerical one.

The error estimator is obtained as a sum of local indicators expressed on each element of the mesh (Mishra, 2012). We
have two types of computable error indicators, the first being linked to the time discretization and the second to the space
discretization.

We say that the a posteriori error estimates are optimal if we are able to bound each one of this indicators by the local
\newpage
error of the solution around the corresponding element. In this work, we propose a low cost discretization relying on the
Crank Nicolson's scheme in time combined with the finite elements in space, and then prove a priori and a posteriori error
estimates for the discrete problem.

The outline of the paper is as follows. In section 2, we give some notations that will be used in the sequel. Section 3 is
devoted to study the discrete problem and the uniqueness of its solution. In section 4, we study the a priori errors and
derive optimal estimates. Section 5 is devoted to study the a posteriori errors where two types of error indicators are
established.

\section{Notations}

In this section we will introduce some notations that will be used in the sequel.

\begin{itemize}[itemsep=0.67em,topsep=5pt]
\item $h$ the maximal diameter of the elements of all $\tau_{nh}$
\item $h_n$ the maximal diameter of the elements of $\tau_{nh}$ for each $n$
\item $h_\kappa$ the diameter of $\kappa$
\item $h_e$ the diameter of the edge $e$
\item $\Delta_\kappa$ the union of elements of $\tau_{nh}$ that intersect $\kappa$
\item $\Delta_e$ the union of elements of $\tau_{nh}$ that intersect the edge $e$
\item $\epsilon_\kappa$ the set of edges of $\kappa$ that are not on $\Gam$
\item $\epsilon_\kappa^m$ the set of edges of $\kappa$ that are on $\Gam$
\item $[.]_e$ the jump through $e$ for each edge $e$ in $\epsilon_\kappa$
\item $\psi_\kappa$ the bubble function which is equal to the product of the three barycentric coordinates associated
with the vertices of $\kappa$
\item $L_e$ the lifting operator defined on polynomials on $e$ vanishing on $\partial e$
\item $X_{nh}$ the finite dimensional space of functions such that their restrictions to any element $\kappa$ of
$\tau_{nh}$ belong to a space of polynomials of degree one. In other words,
\[
X_{nh}=\left\{v_n^h\in C^0(\overline\Om),v_n^h|_\kappa
isa f f ine\forall\kappa\in T_{nh}\right\}
\]
\item $I_h$ the approximation operator in $\mathcal L(H^2(\Om);X_{nh})$ such that for $m=0,1$,
\[
\forall v\in H^2(\Om),\qquad
\abs{I_h(v)-v}_{m,\Om}\le Ch^{2-m}\abs{v}_{2,\Om}
\]
\item We introduce the Sobolev spaces:
\[
H^m(\Om)=\left\{v\in L^2(\Om),\partial^\alpha v\in L^2(\Om),\ \forall\abs{\alpha}\le m\right\},
\]
equipped with the following semi-norm and norm:
\[
\abs{v}_{m,\Om}=\left\{\sum_{\abs{\alpha}=m}\int_\Om
\abs{\partial^\alpha v(x)}^2\dd x\right\}^{1/2}
\]
and
\[
\norm{v}{m,\Om}=\left\{\sum_{k\le m}\abs{v}_{k,\Om}^2\right\}^{1/2}
\]
\end{itemize}

\newpage

\section{The Discrete Problem}

Assume that $\Om$ is a polyhedron and $\gamma$ denotes a positive smooth bounded function. We introduce a partition
of the interval $[0,T]$ into sub intervals $[t_{n-1},t_n]$, $1\le n\le N$, such that
$0=t_0\le t_1\le...\le t_N=T$. Denote by $\tau_n$ the length of $[t_{n-1},t_n]$, by $\abs{\tau}$ the maximum
of the $\tau_n$, by $\tau$ the $N$-tuple $(\tau_1,...,\tau_N)$, and by $\sigma_\tau$ the regularity parameter
\[
\sigma_\tau=\max_{2\le n\le N}\frac{\tau_n}{\tau_{n-1}}.
\]

\begin{theorem}
If $u(t)\in H^2(\Om)$, then Problem~\eqref{eq:continuous-problem} is equivalent to the variational problem,
\begin{equation}
\begin{cases}
\text{Find }u(t)\in H^1(\Om),\text{ such that}\\
u(0,x)=u_0(x),&\text{on }\Gam\\
\displaystyle \int_\Om\gamma\nabla u\nabla v\dd x
+\int_\Gam\frac{\partial u}{\partial t}(t,s)v(t,s)\dd s=0,
&\forall v(t)\in H^1(\Om)
\end{cases}
\label{eq:variational-problem}
\end{equation}
\end{theorem}

\begin{proof}
Let $u(t)$ be a solution of problem~\eqref{eq:continuous-problem}. Multiplying the first equation of
problem~\eqref{eq:continuous-problem} by $v(t)\in H^1(\Om)$, integrating over $\Om$, applying Green's formula and
using the second equation of problem~\eqref{eq:continuous-problem}, we obtain that $u$ is also a solution of
problem~\eqref{eq:variational-problem}. Conversely, if $u$ is a solution of problem~\eqref{eq:variational-problem},
we take $v(t)\in\mathcal D(\Om)$ to get the first line of problem~\eqref{eq:continuous-problem}. Then multiplying the
first equation of problem~\eqref{eq:continuous-problem} by $v(t)\in H^1(\Om)$, integrating over $\Om$, using the Green's
formula and comparing with problem~\eqref{eq:variational-problem}, we get the second line of
problem~\eqref{eq:continuous-problem}.
\end{proof}

\begin{proposition}
The solution of Problem~\eqref{eq:variational-problem} satisfies the following bound:
\[
\norm{u}{L^\infty(0,T,L^2(\Gam))}^2\le \norm{u_0}{L^2(\Gam)}
\]
\end{proposition}

Now, the full discrete problem associated to the variational problem~\eqref{eq:variational-problem} is:
\begin{equation}
\begin{cases}
\text{Given }u_h^n\in X_{nh},\\
\forall v_h(t)\in X_{nh},\ u_h^n(t)\text{ is the solution of}\\
\displaystyle
\int_\Om\gamma\nabla u_h^{n+1}(x)\nabla v_h(t,x)\dd x
+2\int_\Gam\frac{u_h^{n+1}-u_h^n}{\tau_n}(x)v_h(t,x)\dd x
-\int_\Om\gamma\nabla u_h^n\nabla v_h\dd x=0
\end{cases}
\label{eq:discrete-problem}
\end{equation}

\begin{theorem}
The problem~\eqref{eq:discrete-problem} admits a unique solution in $X_{nh}$.
\end{theorem}

\begin{proof}
We introduce the bilinear form ,
\[
a(u_h^{n+1},v_h)=\int_\Om \tau_n\gamma\nabla u_h^{n+1}\nabla v_h\dd x
+2\int_\Gam u_h^{n+1}v_h\dd\sigma
\]
and the linear form
\[
L(v_h)=\int_\Om \tau_n\gamma\nabla u_h^n\nabla v_h\dd x
+2\int_\Gam u_h^nv_h\dd\sigma
\]
Then the previous problem can be written as
\[
\forall v_h\in X_{nh},\qquad a(u_h^{n+1},v_h)=L(v_h)
\]
It is obvious that $a$ is bilinear and continuous in $X_{n+1,h}\times X_{n+1,h}$, and that $L$ is linear and continuous
in $X_{nh}$ and then, the Lax-Milgram theorem states the existence and the uniqueness of the solution. See
(Arwadi, Dib \& Sayah, 2015).
\end{proof}

\section{A Priori Error Estimate}

To get an a priori error estimate, we need the following Gronwall's lemma.

\begin{lemma}[Gronwall's lemma:]
Let $(a_n)_n\ge0$, $(b_n)_n\ge0$ and $(c_n)_n\ge0$ be three real positive sequences such that
$(c_n)_n\ge0$ is an increasing sequence. Suppose that
\[
a_0+b_0\le c_0
\]
there exists $\lambda>0$ such that:
\[
\forall n\ge0,\qquad a_n+b_n\le c_n+\lambda\sum_{m=0}^{n-1}a_m
\]
\newpage
then we have
\[
\forall n\ge0,\qquad a_n+b_n\le c_ne^{n\lambda}
\]
\end{lemma}

\begin{theorem}
If $u\in L^\infty$ we have,
\[
\norm{u(t_{m+1})-u_h^{m+1}}{0,\Gam}^2
+k\abs{C_\gamma}\sum_{n=0}^{m}\norm{u(t_{n+1})-u_h^{n+1}}{1,\Om}^2
\le c(h^2+k^2)
\]
where $c$ is a constant independent of $h$ and $k$.
\end{theorem}

\begin{proof}
Denote by $k$ the time step, $h$ the parameter of the mesh and $X_h$ the discrete space. Suppose that $\tau_n$ and
$\tau_{nh}$ are constants during time iterations. Consider the equation,
\[
\int_\Om\nabla u(t,x)\nabla v(t,x)\dd x
+2\int_\Gam\frac{\partial u}{\partial t}(t,s)v(t,s)\dd s=0,
\qquad \forall v(t)\in H^1(\Om)
\]
For $t\in(t_n,t_{n+1})$ take $v=v_h^{n+1}$, integrate in time
\begin{equation}
\int_{t_n}^{t_{n+1}}\int_\Om\nabla u(t,x)\nabla v_h^{n+1}(t,x)\dd x
+2\int_{t_n}^{t_{n+1}}\int_\Gam
\frac{\partial u}{\partial t}(t,s)v_h^{n+1}(t,s)\dd s=0
\label{eq:time-integrated-continuous}
\end{equation}
The discrete variation formulation for the Crank Nicolson scheme taken in the time step $n+1$, is
\[
\int_\Om\gamma\nabla u_h^{n+1}\nabla v_h^{n+1}\dd x
+2\int_\Gam\frac{u_h^{n+1}-u_h^n}{\tau_n}v_h^{n+1}\dd\sigma
-\int_\Om\gamma\nabla u_h^n\nabla v_h^{n+1}\dd x=0
\]
Integrating in time between $t_n$ and $t_{n+1}$ we get,
\begin{equation}
\int_{t_n}^{t_{n+1}}\int_\Om\gamma\nabla u_h^{n+1}\nabla v_h^{n+1}\dd x
+2\int_{t_n}^{t_{n+1}}\int_\Gam
\frac{u_h^{n+1}-u_h^n}{\tau_n}v_h^{n+1}\dd\sigma
-\int_{t_n}^{t_{n+1}}\int_\Om\gamma\nabla u_h^n\nabla v_h^{n+1}\dd x=0
\label{eq:time-integrated-discrete}
\end{equation}
Taking the difference between \eqref{eq:time-integrated-continuous} and
\eqref{eq:time-integrated-discrete} we get,
\[
\int_{t_n}^{t_{n+1}}\int_\Om
\gamma\nabla(u-u_h^{n+1}+u_h^n)\nabla v_h^{n+1}\dd x
+2\int_\Gam\bigl[(u(t_{n+1})-u(t_n))-(u_h^{n+1}-u_h^n)\bigr]v_h^{n+1}\dd\sigma=0
\]
Now inserting $\pm\nabla(I_h(u(t_{n+1})))$, $\nabla(I_h(u(t_n)))$,
$\nabla(u(t_{n+1}))$ and $\nabla(u(t_n))$ into the first term, and
$\pm I_h(u(t_{n+1}))$ and $I_h(u(t_n))$ into the second term, we obtain
\begin{align*}
&-\int_{t_n}^{t_{n+1}}\int_\Om
\gamma\nabla(u(t_{n+1})-u(t_n))\nabla v_h^{n+1}\dd x\dd t
-\int_{t_n}^{t_{n+1}}\int_\Om
\gamma\nabla(I_h(u(t_{n+1}))-u(t_{n+1}))\nabla v_h^{n+1}\dd x\dd t\\
&+\int_{t_n}^{t_{n+1}}\int_\Om
\gamma\nabla(I_h(u(t_{n+1}))-u_h^{n+1})\nabla v_h^{n+1}\dd x\dd t
-\int_{t_n}^{t_{n+1}}\int_\Om
\gamma\nabla(I_h(u(t_n))-u_h^n)\nabla v_h^{n+1}\dd x\dd t\\
&+\int_{t_n}^{t_{n+1}}\int_\Om
\gamma\nabla(I_h(u(t_n))-u(t_n))\nabla v_h^{n+1}\dd x\dd t
+\int_{t_n}^{t_{n+1}}\int_\Om
\gamma\nabla(u(t_n))\nabla v_h^{n+1}\dd x\dd t\\
&+2\int_\Gam(a_{n+1}-a_n)(s)v_h^{n+1}\dd s
-2\int_\Gam(I_h(u(t_{n+1}))-u(t_{n+1}))
 -(I_h(u(t_n))-u(t_n))v_h^{n+1}\dd s=0
\end{align*}
where $a_{n+1}=I_h(u(t_{n+1}))-u_h^{n+1}$ and $a_n=I_h(u(t_n))-u_h^n$.
Now we will bound the third and fourth terms of the previous equation. Choosing $v_h^{n+1}=a_{n+1}$
\begin{align*}
\int_{t_n}^{t_{n+1}}\int_\Om
\gamma\nabla(I_h(u(t_{n+1}))-u_h^{n+1})\nabla v_h^{n+1}\dd x\dd t
&=\int_{t_n}^{t_{n+1}}\int_\Om\gamma\nabla a_{n+1}\nabla v_h^{n+1}\dd x\dd t\\
&=\int_{t_n}^{t_{n+1}}\int_\Om\gamma\nabla a_{n+1}^2\dd x\dd t\\
&\le k\abs{C_\gamma}\abs{a_{n+1}}_{1,\Om}^2
\end{align*}
\newpage
\begin{align*}
\int_{t_n}^{t_{n+1}}\int_\Om
\gamma\nabla(I_h(u(t_n))-u_h^n)\nabla v_h^{n+1}\dd x\dd t
&=\int_{t_n}^{t_{n+1}}\int_\Om\gamma\nabla a_n\nabla a_{n+1}\dd x\dd t\\
&\le k\abs{C_\gamma}\abs{a_n}_{1,\Om}\abs{a_{n+1}}_{1,\Om}
\end{align*}
we obtain
\begin{align*}
&2\int_\Gam(a_{n+1}-a_n)(s)v_h^{n+1}\dd s
+k\abs{C_\gamma}\abs{a_{n+1}}_{1,\Om}^2
-k\abs{C_\gamma}\abs{a_n}_{1,\Om}\abs{a_{n+1}}_{1,\Om}\\
={}&2\int_\Gam(I_h(u(t_{n+1}))-u(t_{n+1}))
 -(I_h(u(t_n))-u(t_n))v_h^{n+1}\dd s\\
&+\int_{t_n}^{t_{n+1}}\int_\Om
\gamma\nabla(u(t_{n+1})-u(t))\nabla v_h^{n+1}\dd x\dd t\\
&+\int_{t_n}^{t_{n+1}}\int_\Om
\gamma\nabla(I_h(u(t_{n+1}))-u(t_{n+1}))\nabla v_h^{n+1}\dd x\dd t\\
&-\int_{t_n}^{t_{n+1}}\int_\Om
\gamma\nabla(I_h(u(t_n))-u(t_n))\nabla v_h^{n+1}\dd x\dd t\\
&-\int_{t_n}^{t_{n+1}}\int_\Om\gamma\nabla(u(t_n))\nabla v_h^{n+1}\dd x\dd t
\end{align*}
We denote by $T_1$ the first term of the left hand side, $T_2$ and $T_3$ the first and second terms of the
right hand side, $T_4$ the third and fourth terms, and $T_5$ the last term of the equation.

The term $T_1$ can be expressed as
\begin{align*}
T_1&=2\int_\Gam(a_{n+1}-a_n)(s)v_h^{n+1}\dd s
=2\int_\Gam(a_{n+1}^2-a_na_{n+1})\dd s\\
&=\int_\Gam a_{n+1}^2\dd s-\int_\Gam a_n^2\dd s
+\int_\Gam(a_{n+1}-a_n)^2\dd s
\end{align*}

The term $T_2$ can be bounded as
\begin{align*}
T_2
&=2\int_\Gam(I_h(u(t_{n+1}))-u(t_{n+1}))
 -(I_h(u(t_n))-u(t_n))v_h^{n+1}\dd s\\
&=2\int_\Gam(g(t_{n+1})-g(t_n))a_{n+1}\dd s\\
&=2\int_{t_n}^{t_{n+1}}\int_\Om g'(\tau,s)a_{n+1}\dd s\dd\tau\\
&\le2\int_{t_n}^{t_{n+1}}\norm{g'(\tau)}{0,\Gam}\norm{a_{n+1}}{0,\Gam}\dd\tau
\end{align*}
But
\[
\norm{g'(\tau)}{0,\Gam}\le c\norm{g'(\tau)}{1,\Om}
\le\widetilde c h\norm{u'(\tau)}{2,\Om}
\le c_1h\norm{u'(\tau)}{L^\infty(0,T,H^2(\Om))}
\]
then,
\[
T_2\le c_1hk\norm{u'(\tau)}{L^\infty(0,T,H^2(\Om))}\norm{a_{n+1}}{0,\Gam}
\]
Using the inequality $ab\le \frac{1}{2\epsilon_1}a^2+\frac{\epsilon_1}{2}b^2$, with
$a=c_1h\sqrt{k}\norm{u'}{L^\infty}$ and
$b=\sqrt{k}\norm{a_{n+1}}{0,\Gam}$, we get
\[
T_2\le \frac{1}{2\epsilon_1}c_1^2h^2k
\norm{u'(\tau)}{L^\infty(0,T,H^2(\Om))}^2
+\frac{\epsilon_1}{2}k\norm{a_{n+1}}{0,\Gam}^2
\]

\newpage

The term $T_3$ can be bounded as
\begin{align*}
T_3
&=\int_{t_n}^{t_{n+1}}\int_\Om
\gamma\nabla(u(t_{n+1})-u(t))\nabla v_h^{n+1}\dd x\dd t\\
&=\int_{t_n}^{t_{n+1}}\int_t^{t_{n+1}}\int_\Om
\gamma\nabla(u'(\tau,x))\nabla a_{n+1}\dd x\dd\tau\dd t\\
&\le\int_{t_n}^{t_{n+1}}\int_t^{t_{n+1}}
\gamma\norm{u'(\tau)}{1,\Om}\norm{a_{n+1}}{1,\Om}\dd\tau\dd t\\
&\le k^2\abs{C_\gamma}\norm{u'}{L^\infty(0,T,H^1(\Om))}\abs{a_{n+1}}_{1,\Om}
\end{align*}
Using the inequality $ab\le \frac{1}{2\epsilon_2}a^2+\frac{\epsilon_2}{2}b^2$, with
$a=k^{3/2}C_\gamma\norm{u'}{L^\infty}$ and $b=\sqrt{k}\norm{a_{n+1}}{1,\Om}$, we get
\[
T_3\le\frac{1}{2\epsilon_2}k^3C_\gamma^2
\norm{u'}{L^\infty(0,T,H^1(\Om))}^2
+\frac{\epsilon_2}{2}k\abs{a_{n+1}}_{1,\Om}^2
\]

Now the term $T_4$ can be bounded as
\begin{align*}
T_4
&=\int_{t_n}^{t_{n+1}}\int_\Om
\gamma\nabla(I_h(u(t_{n+1}))-u(t_{n+1}))\nabla v_h^{n+1}\dd x\dd t\\
&\quad-\int_{t_n}^{t_{n+1}}\int_\Om
\gamma\nabla(I_h(u(t_n))-u(t_n))\nabla a_{n+1}\dd x\dd t\\
&=\int_{t_n}^{t_{n+1}}\int_\Om
\gamma\nabla(g(t_{n+1})-g(t_n))\nabla a_{n+1}\dd x\dd t\\
&=\int_{t_n}^{t_{n+1}}\int_t^{t_{n+1}}\int_\Om
\gamma\nabla(g'(\tau,x))\nabla a_{n+1}\dd x\dd\tau\dd t\\
&\le k^2\abs{C_\gamma}\norm{g'(\tau)}{1,\Om}\abs{a_{n+1}}_{1,\Om}\\
&\le k^2\abs{C_\gamma}ch\norm{u'(\tau)}{2,\Om}\abs{a_{n+1}}_{1,\Om}\\
&\le k^2ch\abs{C_\gamma}\norm{u'}{L^\infty(0,T,H^2(\Om))}\abs{a_{n+1}}_{1,\Om}
\end{align*}
Using the inequality $ab\le \frac{1}{2\epsilon_3}a^2+\frac{\epsilon_3}{2}b^2$, with
$a=\abs{C_\gamma}chk^{3/2}\norm{u'}{L^\infty}$ and
$b=\sqrt{k}\norm{a_{n+1}}{1,\Om}$, we get
\[
T_4\le\frac{1}{2\epsilon_3}\abs{C_\gamma}^2c^2h^2k^3
\norm{u'}{L^\infty(0,T,H^2(\Om))}^2
+\frac{\epsilon_3}{2}k\abs{a_{n+1}}_{1,\Om}^2
\]

Finally, the term $T_5$ can be bounded as
\begin{align*}
T_5&=\int_{t_n}^{t_{n+1}}\int_\Om\gamma\nabla(u(t_n))\nabla v_h^{n+1}\dd x\dd t\\
&\le k\abs{C_\gamma}\norm{u}{1,\Om}\abs{a_{n+1}}_{1,\Om}\\
&\le k\abs{C_\gamma}\norm{u}{L^\infty(0,T,H^1(\Om))}\abs{a_{n+1}}_{1,\Om}
\end{align*}
Using the inequality $ab\le \frac{1}{2\epsilon_4}a^2+\frac{\epsilon_4}{2}b^2$, with
$a=\abs{C_\gamma}k^{1/2}\norm{u}{L^\infty}$ and
$b=k^{1/2}\norm{a_{n+1}}{1,\Om}$, we get
\[
T_5\le\frac{1}{2\epsilon_4}\abs{C_\gamma}^2k
\norm{u}{L^\infty(0,T,H^1(\Om))}^2
+\frac{\epsilon_4}{2}k\norm{a_{n+1}}{1,\Om}^2
\]

Now using all the previous bounds, we obtain
\begin{align*}
&\int_\Gam a_{n+1}^2\dd s-\int_\Gam a_n^2\dd s
+\int_\Gam(a_{n+1}-a_n)^2\dd s
+k\abs{C_\gamma}\abs{a_{n+1}}_{1,\Om}^2
-k\abs{C_\gamma}\abs{a_n}_{1,\Om}\abs{a_{n+1}}_{1,\Om}\displaybreak[4]\\
\le{}&\frac{1}{2\epsilon_1}c_1^2h^2k
\norm{u'(\tau)}{L^\infty(0,T,H^2(\Om))}^2
+\frac{\epsilon_1}{2}k\norm{a_{n+1}}{0,\Gam}^2\\
&+\frac{1}{2\epsilon_2}k^3\abs{\gamma}^2
\norm{u'}{L^\infty(0,T,H^1(\Om))}^2
+\frac{\epsilon_2}{2}k\abs{a_{n+1}}_{1,\Om}^2\\
&+\frac{1}{2\epsilon_3}\abs{C_\gamma}^2c^2h^2k^3
\norm{u'}{L^\infty(0,T,H^2(\Om))}^2
+\frac{\epsilon_3}{2}k\abs{a_{n+1}}_{1,\Om}^2\\
&-\frac{1}{2\epsilon_4}\abs{C_\gamma}^2k
\norm{u}{L^\infty(0,T,H^1(\Om))}^2
-\frac{\epsilon_4}{2}k\norm{a_{n+1}}{1,\Om}^2
\end{align*}
Choosing $\epsilon_1=\frac{1}{8T}$, $\epsilon_2=\frac{\abs{C_\gamma}}{2}$,
$\epsilon_3=\frac{\abs{C_\gamma}}{2}$ and $\epsilon_4=\frac{\abs{C_\gamma}}{2}$, we get
\begin{align*}
&\int_\Gam a_{n+1}^2\dd s-\int_\Gam a_n^2\dd s
+\int_\Gam(a_{n+1}-a_n)^2\dd s
+\frac{3}{4}k\abs{C_\gamma}\abs{a_{n+1}}_{1,\Om}^2
-k\abs{C_\gamma}\abs{a_n}_{1,\Om}\abs{a_{n+1}}_{1,\Om}\\
&\hspace{35mm}\le ck(h^2+k^2)+\frac{k}{16T}\abs{a_{n+1}}_{0,\Gam}
\end{align*}
Taking sum from $n=0,1,\ldots,m$ and replacing $A_m=4\norm{a_{m+1}}{0,\Gam}^2$,
$C_m=4c'(h^2+k^2)$ and
$B_m=4k\abs{\gamma}\sum_{n=0}^{m}
\left(\frac{3}{4}\abs{a_{n+1}^2}-\abs{a_n}\abs{a_{n+1}}\right)$
with $\frac{k}{16T}\le\frac14$,
we get
\[
A_m+B_m\le C_m+\lambda\sum_{n=0}^{m-1}A_n
\]
Using Gronwall's Lemma and the properties of $I_h$ we obtain the result.
\end{proof}

\section{A Posteriori Error Estimate}

In this section a posteriori error estimates between the exact solution and the numerical one will be established.

\begin{proposition}[(Verfurth, 1996)]
Denote by $P_r(\kappa)$ the space of polynomials of degree less than $r$ on $\kappa$, we have $\forall v\in P_r(\kappa)$
\[
c\norm{v}{0,\kappa}\le \norm{v\psi_\kappa^{1/2}}{0,\kappa}
\le c'\norm{v}{0,\kappa}
\qquad
\abs{v}_{1,\kappa}\le ch_\kappa^{-1}\norm{v}{0,\kappa}
\]
\end{proposition}

\begin{proposition}[(Verfurth, 1996)]
Denote by $P_r(e)$ the space of polynomials of degree less than $r$ on $e$, we have $\forall v\in P_r(e)$,
\[
c\norm{v}{0,e}\le \norm{v\psi_e^{1/2}}{0,e}\le c'\norm{v}{0,e}
\]
and for all polynomials in $P_r(e)$ vanishing on $\partial e$,
\[
\norm{L_ev}{0,\kappa}+h_e\abs{L_ev}_{1,\kappa}
\le ch_e^{1/2}\norm{v}{0,e}
\]
\end{proposition}

For the a posteriori error estimates, consider $\forall t\in(t_{n-1},t_n)$ the piecewise affine function $u_h(t)$
which take the values
\[
u_h(t)=\frac{t-t_{n-1}}{\tau_n}(u_h^n-u_h^{n-1})+u_h^{n-1}
\]

The solutions of Problems~\eqref{eq:variational-problem} and~\eqref{eq:discrete-problem} verify the following
\begin{align*}
T(v)
&=\int_\Om\gamma\nabla(u-u_h)\nabla v(t,x)\dd x
+2\int_\Gam\frac{\partial(u-u_h)}{\partial t}(t,x)v(t,x)\dd x\\
&=\int_\Om\gamma\nabla u\nabla v\dd x
-\int_\Om\gamma\nabla u_h\nabla v\dd x
-2\int_\Gam\frac{\partial u_h}{\partial t}(t,x)v(t,x)\dd x
+2\int_\Gam\frac{\partial u}{\partial t}v\dd x\\
&=-\int_\Om\gamma\nabla u_h\nabla v\dd x
-2\int_\Gam\frac{\partial u_h}{\partial t}(t,x)v(t,x)\dd x
\end{align*}
adding and subtracting $u_h^n$ and $u_h^{n-1}$ to the first term, then using the value of $u_h$ we get
\newpage
\begin{align*}
T(v)
&=-\int_\Om\gamma\nabla(u_h-u_h^n)\nabla v\dd x
-\int_\Om\gamma\nabla(u_h^n-u_h^{n-1})\nabla v\dd x
-\int_\Om\nabla u_h^{n-1}\nabla v\dd x
-2\int_\Gam\frac{u_h^n-u_h^{n-1}}{\tau_n}(t,x)v(t,x)\dd x\\
&=-\frac{t-t_n}{\tau_n}\int_\Om\gamma\nabla(u_h^n-u_h^{n-1})\nabla v\dd x
-\int_\Om\gamma\nabla(u_h^n-u_h^{n-1})\nabla v\dd x
-\int_\Om\gamma\nabla u_h^{n-1}\nabla v\dd x
-2\int_\Gam\frac{u_h^n-u_h^{n-1}}{\tau_n}(t,x)v(t,x)\dd x\\
&=-\frac{t-t_{n-1}}{\tau_n}\int_\Om\gamma\nabla(u_h^n-u_h^{n-1})\nabla v\dd x
-\int_\Om\gamma\nabla u_h^{n-1}\nabla v\dd x
-2\int_\Gam\frac{u_h^n-u_h^{n-1}}{\tau_n}(t,x)v(t,x)\dd x
\end{align*}
adding and subtracting $v_h$ to the second and third terms, we get
\begin{align*}
T(v)
&=-\frac{t-t_{n-1}}{\tau_n}\int_\Om\gamma\nabla(u_h^n-u_h^{n-1})\nabla v\dd x
-\int_\Om\gamma\nabla u_h^{n-1}\nabla(v-v_h)\dd x
-\int_\Om\gamma\nabla u_h^{n-1}\nabla v_h\dd x\\
&\quad-2\int_\Gam\frac{u_h^n-u_h^{n-1}}{\tau_n}(v-v_h)\dd x
-2\int_\Gam\frac{u_h^n-u_h^{n-1}}{\tau_n}v_h\dd x\\
&=-\frac{t-t_{n-1}}{\tau_n}\int_\Om\gamma\nabla(u_h^n-u_h^{n-1})\nabla v\dd x
-\int_\Om\gamma\nabla u_h^{n-1}\nabla(v-v_h)\dd x
-\int_\Om\gamma\nabla u_h^n\nabla v_h\dd x\\
&\quad-2\int_\Gam\frac{u_h^n-u_h^{n-1}}{\tau_n}v_h\dd x
\end{align*}
adding and subtracting $v$ to the third term,
\begin{align*}
T(v)
&=-\frac{t-t_{n-1}}{\tau_n}\int_\Om\gamma\nabla(u_h^n-u_h^{n-1})\nabla v\dd x
-\int_\Om\gamma\nabla u_h^{n-1}\nabla(v-v_h)\dd x\\
&\quad-\int_\Om\gamma\nabla u_h^n\nabla(v-v_h)\dd x
-2\int_\Gam\frac{u_h^n-u_h^{n-1}}{\tau_n}(v-v_h)\dd x
+\int_\Om\gamma\nabla u_h^n\nabla v\dd x
\end{align*}
Applying Green's theorem on the second and third terms we get
\begin{align*}
T(v)
={}&-\frac{t-t_{n-1}}{\tau_n}\int_\Om\gamma\nabla(u_h^n-u_h^{n-1})\nabla v\dd x
-\sum_{k\in\tau_{nh}}\left(
\int_k\diver(\gamma\nabla u_h^{n-1})(v-v_h)\dd x
-\int_{\partial\kappa}(\nabla u_h^{n-1}.n)(v-v_h)\dd x\right)\\
&-\sum_{k\in\tau_{nh}}\left(
\int_k\diver(\gamma\nabla u_h^n)(v-v_h)\dd x
-\int_{\partial\kappa}(\nabla u_h^n.n)(v-v_h)\dd x\right)\\
&+\int_\Om\gamma\nabla u_h^n\nabla\dd x
-2\int_\Gam\frac{u_h^n-u_h^{n-1}}{\tau_n}(v-v_h)\dd x\\
={}&-\frac{t-t_{n-1}}{\tau_n}\int_\Om\gamma\nabla(u_h^n-u_h^{n-1})\nabla v\dd x
-\sum_{k\in\tau_{nh}}\int_{\partial\kappa}
(\nabla u_h^{n-1}.n)(v-v_h)\dd x\\
&-\sum_{k\in\tau_{nh}}\int_{\partial\kappa}
(\nabla u_h^n.n)(v-v_h)\dd x
+\int_\Om\gamma\nabla u_h^n\nabla v\dd x
-2\int_\Gam\frac{u_h^n-u_h^{n-1}}{\tau_n}(v-v_h)\dd x
\end{align*}

We define, for every edge $e$ of the mesh, the function
\[
(\phi_{h,n}^{e})=
\begin{cases}
\gamma[\nabla u_h^n.n]_e, & e\in\epsilon_k,\\[0.3em]
\gamma\nabla(u_h^n).n
+\gamma\nabla(u_h^{n-1}).n
+2\dfrac{u_h^n-u_h^{n-1}}{\tau_n}, & e\in\epsilon_k^m.
\end{cases}
\]
We get the following equation
\begin{align*}
&\int_\Om\gamma\nabla(u-u_h)\nabla v(t,x)\dd x
+2\int_\Gam\frac{\partial u-u_h}{\partial t}(t,x)v(t,x)\dd x\displaybreak[4]\\
&\quad=\frac{t_{n-1}-t}{\tau_n}\int_{\Om}
\gamma\nabla(u_h^n-u_h^{n-1})\nabla v\dd x
+\int_{\Om}\gamma\nabla u_h^{n-1}\nabla v\dd x
-\sum_{k\in\tau_{h,n}}\sum_{e\in\partial k}
\int_e\phi_{h,n}^e(x)(v-v_h)\dd x
\end{align*}

For each $k$ in $\tau_{h,n}$ we introduce the indicators
\[
\eta_{n,k}^{\tau}
=\sqrt{\frac{\tau_n}{3}}\norm{\nabla(u_h^n-u_h^{n-1})}{0,k}
+\sqrt{\tau_n}\norm{\nabla u_h^{n-1}}{0,k}
\qquad
(\eta_{n,k}^{h})^2
=\sum_{e\in\partial k}h_e\norm{\phi_{h,n}^e}{0,e}^2
\]

\subsection{Upper Bounds of the Error}

\begin{theorem}
For all $m=1,...,N$, we have the following upper bound
\[
\begin{aligned}
&c\norm{\nabla(u-u_h)}{L^2(0,t_m,L^2(\Om))}^2
+\norm{u(t_m)-u_h^m}{0,\Gam}^2\\
&\quad\le c'\Bigg[
 (\eta_{n,k}^{\tau})^2
+\sum_{n=1}^{m}\sum_{k}\tau_n(\eta_{n,k}^{h})^2\\
&\hspace{56mm}+\norm{u_0-u_h^0}{0,\Gam}^2\Bigg]
\end{aligned}
\]
where $c$ is a constant.
\end{theorem}

\begin{proof}
We denote by $L(v)$ the following,
\[
L(v)=\int_\Om\gamma\nabla(u-u_h)\nabla v(t,x)\dd x
+2\int_\Gam\frac{\partial u-u_h}{\partial t}v\dd s
\]
and we define the function $w(t,x)$ by
\[
w(t,x)=e^{-t}(u-u_h)(t,x)
\]
which verifies the equation
\[
\frac{\partial w}{\partial t}+w=e^{-t}\frac{\partial(u-u_h)}{\partial t}
\]
Multiplying $L(v)$ by $e^{-t}$ and taking $w=v$,
\begin{align*}
e^{-t}L(v)
&=\int_\Om\gamma\nabla(e^{-t}(u-u_h))\nabla v\dd x
+2\int_\Gam e^{-t}\frac{\partial u-u_h}{\partial t}v\dd s\\
&=\int_\Om\gamma\nabla w\nabla v\dd x
+2\int_\Gam wv\dd s+2\int_\Gam\frac{\partial w}{\partial t}v\dd s\\
&=\int_\Om\gamma\abs{\nabla w}^2\dd x
+2\int_\Gam w^2\dd s+\int_\Gam\frac{\partial(w^2)}{\partial t}\dd s\\
&\ge \int_\Om\gamma\abs{\nabla w}^2\dd x
+\int_\Gam\frac{\partial(w^2)}{\partial t}\dd s\\
&\ge c\norm{\nabla w}{0,\Om}^2+\int_\Gam\frac{\partial w^2}{\partial t}\dd s
\end{align*}
Note that $e^{-t}\le1$, so $L(w)\le L(u-u_h)$, then we have the following
\[
c\norm{\nabla w}{0,\Om}^2+\int_\Gam\frac{\partial w^2}{\partial t}\dd s
\le \int_\Om\nabla(u-u_h)\nabla(u-u_h)\dd x
+\int_\Gam\frac{\partial(u-u_h)}{\partial t}(u-u_h)\dd s
\]
Integrating in $(t_{n-1},t_n)$, we get
\[
\int_{t_{n-1}}^{t_n}c\norm{\nabla w}{0,\Om}^2\dd t
+\int_{t_{n-1}}^{t_n}\int_\Gam\frac{\partial(w^2)}{\partial t}\dd s\dd t
\le\int_{t_{n-1}}^{t_n}L(u-u_h)\dd t
\]
\[
\int_{t_{n-1}}^{t_n}c\norm{\nabla w}{0,\Om}^2\dd t
+\int_\Gam w^2(t_n,s)\dd s-\int_\Gam w^2(t_{n-1},s)\dd s
\le\int_{t_{n-1}}^{t_n}L(u-u_h)\dd t
\]
Taking the sum from $1$ to $m$, we get
\newpage
\begin{align*}
&c\sum_{n=1}^{m}\int_{t_{n-1}}^{t_n}
\norm{\nabla e^{-t}(u-u_h)}{0,\Om}^2\dd t
-\int_\Gam e^{-2t}\abs{u-u_h}^2(0,s)\dd s\\
&\qquad+\int_\Gam e^{-2t}\abs{u-u_h}^2(t_m,s)\dd s
\le\sum_{n=1}^{m}\int_{t_{n-1}}^{t_n}L(u-u_h)\dd t
\end{align*}
\begin{align*}
e^{-2T}\left[
 c\sum_{n=1}^{m}\left(\int_{t_{n-1}}^{t_n}\norm{\nabla(u-u_h)}{0,\Om}^2\dd t
+\int_\Gam\abs{u-u_h}^2(t_m,s)\dd s\right)\right]
&\le\sum_{n=1}^{m}\int_{t_{n-1}}^{t_n}L(u-u_h)\dd t\\
&\quad+\int_\Gam\abs{u-u_h}^2(0,s)\dd s
\end{align*}
so that
\[
\int_0^{t_m}c\norm{\nabla(u-u_h)}{0,\Om}^2\dd t
+\norm{u(t_m)-u_h^m}{0,\Gam}^2
\le c'\left[
\sum_{n=1}^{m}\int_{t_{n-1}}^{t_n}L(u-u_h)\dd t
+\norm{u_0-u_h^0}{0,\Gam}^2\right]
\]

We decompose $L(v)=L_1(v)+L_2(v)$ and denote $v=u-u_h$.\\
Now we have to bound $L_1(v)$,
\begin{align*}
L_1(v)
&=\frac{t_n-t}{\tau_n}\int_\Om\gamma\nabla(u_h^n-u_h^{n-1})\nabla v\dd x
-\int_\Om\gamma\nabla u_h^{n-1}\nabla v\dd x\\
&=\frac{t_n-t}{\tau_n}\sum_{k\in\tau_{h,n}}
\int_k\gamma\nabla(u_h^n-u_h^{n-1})\nabla v\dd x
-\sum_{k\in\tau_{h,n}}\int_k\gamma\nabla u_h^{n-1}\nabla v\dd x\\
&\le \abs{\frac{t_n-t}{\tau_n}}
\sum_{k\in\tau_{h,n}}c_\gamma
\norm{\nabla(u_h^n-u_h^{n-1})}{0,k}\norm{\nabla v}{0,k}
+\sum_{k\in\tau_{h,n}}c_\gamma
\norm{\nabla v_h^{n-1}}{0,k}\norm{\nabla v}{0,k}
\end{align*}
Integrating in $(t_{n-1},t_n)$, then taking the sum from $1$ to $m$, we get
\begin{align*}
\int_{t_{n-1}}^{t_n}L_1(v)\dd t
&\le\sum_{k\in\tau_{h,n}}
\left[\int_{t_{n-1}}^{t_n}
\left(\abs{\frac{t_n-t}{\tau_n}}c
\norm{\nabla(u_h^n-u_h^{n-1})}{0,k}\right)^2\dd t\right]^{1/2}
\left[\int_{t_{n-1}}^{t_n}\norm{\nabla v}{0,k}^2\dd t\right]^{1/2}\\
&\quad+\sum_{k\in\tau_{h,n}}
\left[\int_{t_{n-1}}^{t_n}c_\gamma
\norm{\nabla v_h^{n-1}}{0,k}^2\dd t\right]^{1/2}
\left[\int_{t_{n-1}}^{t_n}\norm{\nabla v}{0,k}^2\dd t\right]^{1/2}\\
&=\sum_{k\in\tau_{h,n}}
\left[\frac{c_\gamma^2\tau_n}{3}
\norm{\nabla(u_h^n-u_h^{n-1})}{0,k}^2\right]^{1/2}
\left[\int_{t_{n-1}}^{t_n}\norm{\nabla v}{0,k}^2\dd t\right]^{1/2}\\
&\quad+\sum_{k\in\tau_{h,n}}
\left[c_\gamma^2\tau_n\norm{\nabla v_h^{n-1}}{0,k}^2\dd t\right]^{1/2}
\left[\int_{t_{n-1}}^{t_n}\norm{\nabla v}{0,k}^2\dd t\right]^{1/2}\\
&=\left[\int_{t_{n-1}}^{t_n}\norm{\nabla v}{0,k}^2\dd t\right]^{1/2}
\sum_{k\in\tau_{h,n}}c_\gamma\left[
\sqrt{\frac{\tau_n}{3}}\norm{\nabla(u_h^n-u_h^{n-1})}{0,k}
+\sqrt{\tau_n}\norm{\nabla v_h^{n-1}}{0,k}\right]\\
&=\left[\int_{t_{n-1}}^{t_n}\norm{\nabla v}{0,k}^2\dd t\right]^{1/2}
\sum_{k\in\tau_{h,n}}c_\gamma\left[\eta_{n,k}^{\tau}\right]\\
&=\left[\int_{t_{n-1}}^{t_n}\norm{\nabla v}{0,k}^2\dd t\right]^{1/2}
\sum_{k\in\tau_{h,n}}c_\gamma^2
\left[(\eta_{n,k}^{\tau})^2\right]^{1/2}
\end{align*}
Using the inequality $ab\le \frac{1}{2\epsilon_1}a^2+\frac{\epsilon_1}{2}b^2$, with
$a=\left[c_\gamma^2(\eta_{n,k}^{\tau})^2\right]^{1/2}$ and
$b=\left[\int_{t_{n-1}}^{t_n}\norm{\nabla v}{0,k}^2\dd t\right]^{1/2}$,

we get
\[
\int_{t_{n-1}}^{t_n}L_1(v)\dd t
\le\frac{c_\gamma^2}{2\epsilon_1}
\sum_{k\in\tau_{h,n}}(\eta_{n,k}^{\tau})^2
+\frac{\epsilon_1}{2}\norm{\nabla v}{L^2(t_{n-1},t_n,L^2(\Om))}^2
\]
\newpage
Taking sum from $n=1,..,m$, we get
\[
\sum_{n=1}^{m}\int_{t_{n-1}}^{t_n}L_1(v)\dd t
\le c_\gamma\sum_{n=1}^{m}\sum_{k\in\tau_{h,n}}(\eta_{n,k}^{\tau})^2
+\frac{\epsilon_1}{2}\norm{\nabla(u-u_h)}{L^2(0,t_m,L^2(\Om))}^2
\]

Next, we will bound $L_2(v)$, using the following proposition (Clement, 1975) The cl\'ement regularization operator
$R_{n,h}:H^1(\Om)\to X_h$ has the following property,
$\forall k\in\tau_{n,h}$ and $\forall v\in H^1(\Om)$, we have the following
\[
\norm{v-R_{n,h}v}{0,k}\le ch_k\norm{\nabla v}{0,\Delta_k}
\]
and
\[
\norm{v-R_{n,h}v}{0,e}\le ch_e^{1/2}\norm{\nabla v}{0,\Delta_e}
\]
\begin{align*}
L_2(v)
&=-\sum_{k\in\tau_{h,n}}\sum_{e\in\partial k}
\int_e\phi_{h,n}^e(x)(v-v_h)\dd x\\
&\le\sum_{k\in\tau_{h,n}}\sum_{e\in\partial k}
\norm{\phi_{h,n}^e}{0,e}\norm{v(t)-v_h(t)}{0,e}
\end{align*}
Now we take $v_h(t)=R_{n,h}(v(t))$, and use the above preposition to get
\begin{align*}
L_2(v)
&\le\sum_{k\in\tau_{h,n}}\sum_{e\in\partial k}
\norm{\phi_{h,n}^e}{0,e}\norm{v(t)-R_{n,h}(v(t))}{0,e}\\
&\le\sum_{k\in\tau_{h,n}}\sum_{e\in\partial k}
\norm{\phi_{h,n}^e}{0,e}c_2h_e^{1/2}\norm{\nabla v(t)}{0,\Delta_e}
\end{align*}
Using the inequality $\sum ab\le(\sum a^2)^{1/2}(\sum b^2)^{1/2}$, with
$a=h_e^{1/2}\norm{\phi_{h,n}^e}{0,e}$ and $b=\norm{\nabla v(t)}{0,\Delta_e}$, we obtain
\begin{align*}
L_2(v)
&\le c_2\sum_{k\in\tau_{h,n}}
\left[\sum_{e\in\partial k}h_e\norm{\phi_{h,n}^e}{0,e}^2\right]^{1/2}
\left[\sum_{e\in\partial k}\norm{\nabla v(t)}{0,\Delta_e}^2\right]^{1/2}\\
&\le c_2\left[\sum_{k\in\tau_{h,n}}(\eta_{n,k}^{h})^2\right]^{1/2}
\left[\sum_{k\in\tau_{h,n}}\sum_{e\in\partial k}
\norm{\nabla v(t)}{0,\Delta_e}^2\right]^{1/2}\\
&\le c_3\left[\sum_{k\in\tau_{h,n}}(\eta_{n,k}^{h})^2\right]^{1/2}
\norm{\nabla v(t)}{0,\Om}
\end{align*}
Integrating in $(t_{n-1},t_n)$,
\begin{align*}
\int_{t_{n-1}}^{t_n}L_2(v)\dd t
&\le c_3\left[\int_{t_{n-1}}^{t_n}
\sum_{k\in\tau_{h,n}}(\eta_{n,k}^{h})^2\dd t\right]^{1/2}
\left[\int_{t_{n-1}}^{t_n}\norm{\nabla v(t)}{0,\Om}^2\dd t\right]^{1/2}\\
&\le c_3\left[\sum_{k\in\tau_{h,n}}\tau_n(\eta_{n,k}^{h})^2\right]^{1/2}
\left[\norm{\nabla v}{L^2(t_{n-1},t_n,L^2(\Om))}\right]
\end{align*}
Taking the sum from $n=1,\ldots,m$, we get
\[
\sum_{n=1}^{m}\int_{t_{n-1}}^{t_n}L_2(v)\dd t
\le c_3\left[\sum_{n=1}^{m}\sum_{k\in\tau_{h,n}}
\tau_n(\eta_{n,k}^{h})^2\right]^{1/2}
\left[\norm{\nabla v}{L^2(0,t_m,L^2(\Om))}\right]
\]
\newpage
Using $ab\le \frac{1}{2\epsilon_2}a^2+\frac{\epsilon_2}{2}b^2$,we get
\begin{align*}
\sum_{n=1}^{m}\int_{t_{n-1}}^{t_n}L_2(u-u_h)\dd t
&\le\frac{c_3}{2\epsilon_2}\sum_{n=1}^{m}\sum_{k\in\tau_{h,n}}
\tau_n(\eta_{n,k}^{h})^2
+\frac{\epsilon_2}{2}\norm{\nabla(u-u_h)}{L^2(0,t_m,L^2(\Om))}^2\\
&=c_4\sum_{n=1}^{m}\sum_{k\in\tau_{h,n}}
\tau_n(\eta_{n,k}^{h})^2
+\frac{\epsilon_2}{2}\norm{\nabla(u-u_h)}{L^2(0,t_m,L^2(\Om))}^2
\end{align*}
Using the above bounds, and choosing $\epsilon_1=c/2$ and $\epsilon_2=c/2$ we get
\[
\begin{aligned}
&c\norm{\nabla(u-u_h)}{L^2(0,t_m,L^2(\Om))}^2
+\norm{u(t_m)-u_h^m}{0,\Gam}^2\\
&\quad\le c'\Bigg[
\sum_{n=1}^{m}\sum_{k\in\tau_{h,n}}(\eta_{n,k}^{h})^2
+\sum_{n=1}^{m}\sum_{k\in\tau_{h,n}}\tau_n(\eta_{n,k}^{h})^2\\
&\hspace{56mm}+\norm{u_0-u_h^0}{0,\Gam}^2\Bigg]
\end{aligned}
\tag*{$\square$}
\]
\end{proof}

\begin{theorem}
For all $m=1,2,..,N$ we have
\[
\norm{\frac{\partial(u-u_h)}{\partial t}}
{L^2(0,t_m,H^{-1/2}(\Gam))}^2
\le c'\left[
\sum_{n=1}^{m}\sum_{k\in\tau_{h,n}}
\left((\eta_{n,k}^{\tau})^2+\tau_n(\eta_{n,k}^{\tau})^2\right)
+\norm{u_0-u_h^0}{0,\Gam}^2\right]
\]
where $c'$ is a constant.
\end{theorem}

\begin{proof}
Define the functions $r(t,x)\in H^{1/2}(\Gam)$ and
$w(t,x)=e^{-t}(u-u_h)(t,x)\in H^1(\Om)$, and consider the problem
\begin{equation}
\begin{cases}
\diver(\gamma\nabla w(t,x))=0,&\text{in }(0,T)\times\Om,\\
w(t,x)=r(t,x),&\text{on }(0,T)\times\Gam.
\end{cases}
\label{eq:harmonic-w}
\end{equation}
which admits a unique solution $w(t)\in H^1(\Om)$ verifying,
\[
\norm{\nabla w(t)}{0,\Om}\le c_1\norm{r}{1/2,\Gam}
\]
Consider the equation
\begin{align*}
&\int_\Om\gamma\nabla(u-u_h)\nabla v(t,x)\dd x
+2\int_\Gam\frac{\partial u-u_h}{\partial t}v\dd s\\
&\quad=\frac{t_n-t}{\tau_n}\int_\Om\gamma\nabla(u_h^n-u_h^{n-1})\nabla v\dd x
-\int_\Om\gamma\nabla u_h^{n-1}\nabla v\dd x
-\sum_{k\in\tau_{h,n}}\sum_{e\in\partial k}
\int_e\phi_{h,n}^e(x)(v-v_h)\dd x
\end{align*}
Using the inequalities
\[
L_1(v)\le
\abs{\frac{t_n-t}{\tau_n}}
\sum_{k\in\tau_{h,n}}c\norm{\nabla(u_h^n-u_h^{n-1})}{0,k}\norm{\nabla v}{0,k}
+\sum_{k\in\tau_{h,n}}c\norm{\nabla v_h^{n-1}}{0,k}\norm{\nabla v}{0,k}
\]
and
\[
L_2(v)\le c_3\left[\sum_{k\in\tau_{h,n}}
(\eta_{n,k}^{h})^2\right]^{1/2}\norm{\nabla v(t)}{0,\Om}
\]
we get
\begin{align*}
&\int_\Om\gamma\nabla(u-u_h)\nabla v(t,x)\dd x
+2\int_\Gam\frac{\partial(u-u_h)}{\partial t}v\dd s\\
&\quad\le\norm{\nabla v(t)}{0,\Om}\Bigg[
\abs{\frac{t_n-t}{\tau_n}}\sum_{k\in\tau_{h,n}}
c\norm{\nabla(u_h^n-u_h^{n-1})}{0,k}\\
&\hspace{31mm}
+\sum_{k\in\tau_{h,n}}c\norm{\nabla u_h^{n-1}}{0,k}
+c_3\left(\sum_{k\in\tau_{h,n}}(\eta_{n,k}^{h})^2\right)^{1/2}
\Bigg]
\end{align*}
dividing by $\norm{v}{1,\Om}$ then using Cauchy-Schwartz inequality, we get
\begin{align*}
\frac{2}{\norm{v(t)}{1,\Om}}
\int_\Gam\frac{\partial(u-u_h)}{\partial t}v(t,s)\dd s&{}\displaybreak[4]\\
&\le c_\gamma\norm{\nabla(u-u_h)}{0,\Om}
+c_\gamma\abs{\frac{t_{n-1}-t}{\tau_n}}
\left(\sum_\kappa\norm{\nabla(u_h-u_h^{n-1})}{0,k}^2\right)^{1/2}
+c\left(\sum_\kappa(\eta_{n,k}^{h})^2\right)^{1/2}
+\left(\sum_\kappa c_\gamma\norm{\nabla u_h^n}{0,\Om}^2\right)^{1/2}
\end{align*}
For every $r(t)\in H^{1/2}(\Gam)$, consider the harmonic lifting in $v\in H^1(\Om)$ satisfying,
\begin{equation}
\begin{cases}
\diver(\gamma\nabla v(t,x))=0,&\text{in }(0,T)\times\Om,\\
v(t,x)=r(t,x),&\text{on }(0,T)\times\Gam.
\end{cases}
\label{eq:harmonic-v}
\end{equation}
where
\[
\norm{v(t)}{1,\Om}\le c_1\norm{r}{1/2,\Gam}
=c_1\norm{v}{1/2,\Gam}
\]
so
\[
\frac{1}{\norm{v(t)}{1/2,\Gam}}\le\frac{1}{\norm{v(t)}{1,\Om}}
\]
but
\[
\sup_{v\in H^{1/2}(\Gam)}
\frac{\displaystyle\int_\Gam\frac{\partial(u-u_h)}{\partial t}v\dd s}
{\norm{v}{1/2,\Gam}}
=\norm{\frac{\partial(u-u_h)}{\partial t}}{-1/2,\Gam}
\]
therefore, after integrating over $(t_{n-1},t_n)$, and taking sum from $n=1,\ldots,m$ we get
\[
\norm{\frac{\partial(u-u_h)}{\partial t}}
{L^2(0,t_m,H^{-1/2}(\Gam))}^2
\le c'\left[
\sum_{n=1}^{m}\sum_{k\in\tau_{h,n}}
\left((\eta_{n,k}^{\tau})^2+\tau_n(\eta_{n,k}^{\tau})^2\right)
+\norm{u_0-u_h^0}{0,\Gam}^2\right]
\tag*{$\square$}
\]
\end{proof}

\begin{theorem}
For all $m=1,...,N$, we have the following
\[
\norm{\nabla(u-\pi_\tau u_h)}{L^2(0,t_m,L^2(\Om))}^2
\le c\left[
\sum_{n=1}^{m}\sum_{k\in\tau_{h,n}}
\left((\eta_{n,k}^{\tau})^2+\tau_n(\eta_{n,k}^{\tau})^2\right)
+\norm{u_0-u_h^0}{0,\Gam}^2\right]
\]
where $c$ is a constant.
\end{theorem}

\begin{proof}
We have, using the previous theorem, the following bound
\begin{align*}
\norm{\nabla(u-\pi_\tau u_h)}{L^2}
&=\norm{\nabla(u-u_h+u_h-\pi_\tau u_h)}{L^2}\\
&\le\norm{\nabla(u-u_h)}{L^2}
+\norm{\nabla(u_h-\pi_\tau u_h)}{L^2}\\
&\le c'\Bigg[
\sum_{n=1}^{m}\sum_{k\in\tau_{h,n}}
\left((\eta_{n,k}^{\tau})^2
+\tau_n(\eta_{n,k}^{\tau})^2\right)\\
&\hspace{28mm}
+\norm{u_0-u_h^0}{0,\Gam}^2\Bigg]^{1/2}
+\norm{\nabla(u_h-\pi_\tau u_h)}{L^2}
\end{align*}
Now we have to bound $\norm{\nabla(u_h-\pi_\tau u_h)}{L^2}$. For
$t\in(t_{n-1},t_n)$, we have $\pi_\tau u_h(t)=u_h^n$ and
$u_h-u_h^n=\frac{t-t_n}{\tau_n}(u_h^n-u_h^{n-1})$, we have
\begin{align*}
\norm{\nabla(u_h-\pi_\tau u_h)}{0,\Om}^2
&\le\frac{(t-t_n)^2}{\tau_n^2}
\left[\sum_k\norm{\nabla(u_h^n-u_h^{n-1})}{0,k}^2\right]\\
&\le\frac{(t-t_n)^2}{\tau_n^2}
\left[\sum_k\norm{\nabla(u_h^n-u_h^{n-1})}{0,k}^2\right]
+\tau_n^2\norm{\nabla u_h^{n-1}}{0,k}^2
\end{align*}
integrating over $(t_{n-1},t_n)$, we get
\begin{align*}
\int_{t_{n-1}}^{t_n}\norm{\nabla(u_h-\pi_\tau u_h)}{0,\Om}^2\dd t
&\le\int_{t_{n-1}}^{t_n}\frac{(t-t_n)^2}{\tau_n^2}
\left[\sum_k\norm{\nabla(u_h^n-u_h^{n-1})}{0,k}^2\right]
+\tau_n^2\norm{\nabla u_h^{n-1}}{0,k}^2\\
&\le c_1\sum_k(\eta_{n,k}^{\tau})^2
\end{align*}
\newpage
Finally we conclude
\[
\norm{\nabla(u-\pi_\tau u_h)}{L^2(0,t_m,L^2(\Om))}
\le c\left[
\sum_{n=1}^{m}\sum_{k\in\tau_{h,n}}
\left((\eta_{n,k}^{\tau})^2+\tau_n(\eta_{n,k}^{\tau})^2\right)
+\norm{u_0-u_h^0}{0,\Gam}^2\right]^{1/2}
\tag*{$\square$}
\]
\end{proof}

\subsection{Upper Bounds of the Indicators}

\begin{theorem}
For all $m=1,..,N$ we have the following estimate
\[
(\eta_{n,k}^{\tau})^2
\le\norm{\nabla(u-\pi_\tau u_h)}{L^2(t_{n-1},t_n,L^2(k))}^2
+\norm{\nabla(u-u_h)}{L^2(t_{n-1},t_n,L^2(k))}^2
\]
\end{theorem}

\begin{proof}
We have
\[
\frac{t-t_n}{\tau_n}\nabla(u_h^n-u_h^{n-1})
=\nabla(u_h-\pi_\tau u_h)
=\nabla(u-u_h)+\nabla(u-\pi_\tau u_h)
\]
then
\begin{align*}
\abs{\frac{t-t_n}{\tau_n}\nabla(u_h^n-u_h^{n-1})}^2
&=\abs{\nabla(u-u_h)+\nabla(u-\pi_\tau u_h)}\\
&\le\abs{\nabla(u-u_h)}^2+\abs{\nabla(u-\pi_\tau u_h)}^2
\end{align*}
but
\[
\left(\frac{t-t_n}{\tau_n}\right)^2
\abs{\nabla(u_h^n-u_h^{n-1})}^2
\le
\left(\frac{t-t_n}{\tau_n}\right)^2
\abs{\nabla(u_h^n-u_h^{n-1})}^2
+\abs{\nabla u_h^{n-1}}^2
\]
integrating over $k$ and on $(t_{n-1},t_n)$ we get
\begin{align*}
&\int_{t_{n-1}}^{t_n}\int_k
\left(\frac{t-t_n}{\tau_n}\right)^2
\abs{\nabla(u_h^n-u_h^{n-1})}^2
+\int_{t_{n-1}}^{t_n}\int_k\abs{\nabla u_h^{n-1}}^2\\
&\hspace{30mm}\le
\int_{t_{n-1}}^{t_n}\int_k
\left(\abs{\nabla(u-u_h)}^2+\abs{\nabla(u-\pi_\tau u_h)}^2\right)
\end{align*}
then
\[
\begin{aligned}
&\frac{\tau_n}{3}\norm{\nabla(u_h^n-u_h^{n-1})}{0,k}^2
+\tau_n\norm{\nabla u_h^{n-1}}{0,k}^2\\
&\quad\le\norm{\nabla(u-\pi_\tau u_h)}
{L^2(t_{n-1},t_n,L^2(k))}^2
+\norm{\nabla(u-u_h)}
{L^2(t_{n-1},t_n,L^2(k))}^2
\end{aligned}
\]
\end{proof}

\medskip
\noindent\textbf{REMARK:} The numerical simulation will be done in a forthcoming paper.

\section*{References}

\begin{list}{}{\setlength{\leftmargin}{1.6em}\setlength{\itemindent}{-1.6em}
  \setlength{\itemsep}{0.55em}\setlength{\parsep}{0pt}}
\item Cherif, M. A., El Arwadi, T., Emmamirad, H., \& Sac- Epee J. M. (2014).
Dirichlet-to-Neumann semigroup acts as a magnifying glass.
\emph{Semigroup Forum}, \textbf{88}(3), 753-767.

\item Clement, P. (1975).
Approximation by finite element functions using local regularization.
\emph{R.A.I.R.O. Anal. Numer.}, \textbf{9}, 77-84.

\item EL Arwadi, T., DIB, S., \& SAYAH, T. (2015).
A priori and a posteriori analysis for a linear elliptic problem with dynamic boundary condition.
\emph{Appl. Math. Inf. Sci.}, \textbf{9}(6), 3305317.

\item Emmamirad, H., \& Shariftabar, M. (2013).
On Explicit representation and Approximation of Dirichlet-to-Neumann Semigroup.
\emph{Semigroup Forum}, \textbf{86}(1), 192-201.

\item Hecht, F. (2012).
New development in Freefem++.
\emph{Journal of Numerical Mathematics}, \textbf{20}, 251-266.

\item Lax, P. D. (2002).
\emph{Functional Analysis}, Wiley Inter-science, New-York.

\item Mishra, V. N. (2007).
\emph{Some Problems on Approximations of Functions in Banach Spaces},
Ph.D. Thesis, Indian Institute of Technology, Roorkee - 247 667, Uttarakhand, India.

\item Mishra, V. N., \& Mishra, L. N. (2012).
Trigonometric Approximation of Signals (Functions) in Lp$p_1$ norm.
\emph{International Journal of Contemporary Mathematical Sciences}, \textbf{7}(19), 909 C 918.

\newpage

\item Mishra, V. N., Khan, H. H., Khatri, K., \& Mishra, L. N. (2013).
Degree of conjugate of signals (functions) belonging to the generalized weighted Lipschitz class by
$(C,1)(E,q)$ means if conjugate trigonometric Fourier series.
\emph{Bulletin of Mathematical Analysis and Applications}, \textbf{5}(4), 40-53.

\item Ostrowski, A. (1940).
Recherches sur la m\'ethode de Graeffe et les z\'eros des polyn\^omes et des s\'eries des Laurent.
\emph{Acta Math.}, \textbf{72}, 9957.

\item Verfurth, R. (1996).
\emph{A posteriori error estimates and Adaptive Mesh-Refinement Techniques}, Wiley and Teubner mathematics.

\item Vrabie, I. I. (2003).
\emph{$C_0$-Semigroups and applications}, North- Holland, Amsterdam.
\end{list}

\end{document}